\documentclass[11pt]{amsart}

\usepackage{amsmath,amssymb,amsthm,mathtools}
\usepackage{enumitem}
\usepackage{graphicx}
\usepackage{tikz}
\usepackage{placeins}
\usepackage{float}
\usetikzlibrary{arrows.meta}
\usepackage{hyperref}
\usepackage[margin=1.15in]{geometry}

\newtheorem{theorem}{Theorem}[section]
\newtheorem{lemma}[theorem]{Lemma}

\newtheorem{remark}[theorem]{Remark}

\title{On Talagrand Type}

\author{Nika Areshidze}
\address{
Department of Mathematics,
University of California, Irvine,
CA 92697, USA
}
\email{nareshid@uci.edu}



\begin{document}

\begin{abstract}
We prove that for every $1<s<2$ and every $0<\delta<1$, Rademacher type $s$ coincides with Talagrand type $(s,\psi_{s,\delta})$. Together with the result of Cordero-Erausquin and Eskenazis for $s=2$, this yields the equivalence for every $1<s\le2$ and every $0<\delta<1$.
\end{abstract}

\maketitle

\vspace{-2.3em}

\section{Introduction}

Let $\Omega_n=\{-1,1\}^n$ be the $n$-dimensional Hamming cube equipped with the uniform probability measure. For a Banach space $X$ and a function $f:\Omega_n\to X$, we write $\mathbb Ef$ for the average of $f$ over $\Omega_n$. For $j\in[n]$, let
$$
D_jf(\varepsilon)
=
\frac{f(\varepsilon)-f(\varepsilon^{(j)})}{2},
$$
where $\varepsilon^{(j)}$ is obtained from $\varepsilon$ by changing the sign of its $j$-th coordinate.

For scalar-valued functions, the classical Poincar\'e inequality on the Hamming cube states that
$$
\|f-\mathbb Ef\|_{L^2(\Omega_n)}^2
\le
\sum_{j=1}^n
\|D_jf\|_{L^2(\Omega_n)}^2.
$$
A fundamental strengthening of this inequality was obtained by Talagrand \cite{Talagrand}. In particular, he proved that there exists a universal constant $C>0$ such that for every $n\in\mathbb{N}$ and every function $f:\Omega_n\to\mathbb{R}$, one has
$$
\|f-\mathbb Ef\|_{L^2(\Omega_n)}^2
\le C
\sum_{j=1}^n
\frac{\|D_jf\|_{L^2(\Omega_n)}^2}
{1+\log\left(
\|D_jf\|_{L^2(\Omega_n)}
/
\|D_jf\|_{L^1(\Omega_n)}
\right)}.
$$

In the vector-valued setting, inequalities of Poincar\'e type are closely related to geometric properties of the target Banach space. Recall that a Banach space $(X,\|\cdot\|_X)$ has Rademacher type $s$, where $1\le s\le 2$, if there exists a constant $0<T<\infty$ such that for every $n\in\mathbb{N}$ and every $x_1,\ldots,x_n\in X$,
$$
\mathbb{E}_{\varepsilon}
\left\|
\sum_{j=1}^{n}\varepsilon_j x_j
\right\|_{X}^{s}
\le
T^s
\sum_{j=1}^{n}
\|x_j\|^{s},
$$
where $\varepsilon_1,\ldots,\varepsilon_n$ are independent Rademacher random variables. The least such constant is denoted by $T_s^R(X)$.

A nonlinear counterpart of Rademacher type is Enflo type \cite{Enflo}. The equivalence between Rademacher type and Enflo type was proved by Ivanisvili, van Handel, and Volberg \cite{IVV}, settling a longstanding problem. A different proof of this equivalence was later given in \cite{AreshidzeEnflo}. Talagrand type, introduced by Cordero-Erausquin and Eskenazis \cite{CorderoEskenazis}, is a refinement of Enflo type. Thus, for Banach spaces, Talagrand type may be viewed as a stronger nonlinear form of Rademacher type.

Before introducing Talagrand type, recall that a function $\psi:[0,\infty)\to[0,\infty)$ is called a Young function if it is convex and satisfies
$$
\lim_{t\to 0}\frac{\psi(t)}{t}=0,
\qquad
\lim_{t\to\infty}\frac{\psi(t)}{t}=\infty.
$$
For a Banach space $X$ and a function $f:\Omega_n\to X$, the corresponding Orlicz norm is defined by
$$
\|f\|_{L_{\psi}(\Omega_n,X)}
=
\inf\left\{
\lambda>0:
\mathbb{E}_{\varepsilon}
\psi\left(
\frac{\|f(\varepsilon)\|_X}{\lambda}
\right)
\le 1
\right\}.
$$

Motivated by Talagrand's inequality, Cordero-Erausquin and Eskenazis introduced the notion of Talagrand type. They gave the definition for general metric spaces, but for the purposes of this paper we state it only for Banach spaces.

Let $X$ be a Banach space, let $\psi:[0,\infty)\to[0,\infty)$ be a Young function, and let $s\in(0,\infty)$. We say that $X$ has Talagrand type $(s,\psi)$ with constant $\tau\in(0,\infty)$ if for every $n\in\mathbb{N}$ and every $f:\Omega_n\to X$,
$$
\mathbb{E}_{\varepsilon,\eta}
\|f(\varepsilon)-f(\eta)\|_X^s
\le
\tau^s
\sum_{j=1}^n
\|D_jf\|_{L_{\psi}(\Omega_n,X)}^s.
$$
For $1\le s\le2$, Talagrand type $(s,\psi)$ implies
Rademacher type $s$. Indeed, given $x_1,\ldots,x_n\in X$, define $f(\varepsilon)=\sum_{j=1}^n\varepsilon_jx_j.$
Then $\mathbb Ef=0$ and $D_jf(\varepsilon)=\varepsilon_jx_j.$ Hence
$$
\|D_jf\|_{L_\psi(\Omega_n,X)}
=
c_\psi\|x_j\|_X,
$$
where $c_\psi=\|\mathbf 1\|_{L_\psi}$. Moreover, by Jensen's inequality,
$$
\mathbb E_{\varepsilon,\eta}
\|f(\varepsilon)-f(\eta)\|_X^s
\ge
\mathbb E_\varepsilon
\left\|
f(\varepsilon)-\mathbb E_\eta f(\eta)
\right\|_X^s
=
\mathbb E_\varepsilon
\|f(\varepsilon)\|_X^s.
$$
Therefore, Talagrand type $(s,\psi)$ yields
$$
\mathbb E_\varepsilon
\left\|
\sum_{j=1}^n\varepsilon_jx_j
\right\|_X^s
\le
(\tau c_\psi)^s
\sum_{j=1}^n\|x_j\|_X^s,
$$
so $X$ has Rademacher type $s$. They also observed \cite{CorderoEskenazis} that if a Banach space $X$ has the property that for every $n\in\mathbb{N}$ and every $f:\Omega_n\to X$,
$$
\|f-\mathbb{E}f\|_{L^s(\Omega_n,X)}^s
\le
\tau_*^s
\sum_{j=1}^n
\|D_jf\|_{L_{\psi}(\Omega_n,X)}^s
$$
for some $\tau_*\in(0,\infty)$, then $X$ has Talagrand type $(s,\psi)$.

For $s\in(1,\infty)$ and $\delta\in[0,1]$, fix a Young function
$\psi_{s,\delta}:[0,\infty)\to[0,\infty)$ satisfying
$$
\psi_{s,\delta}(t)
=
t^s\log^{-\delta}(e+t)
$$
for all sufficiently large $t$.

Cordero-Erausquin and Eskenazis \cite{CorderoEskenazis} proved that if a Banach space $X$ has Rademacher type $s$, where $1<s\le2$, then for every $0<\delta<s/2$, the space $X$ has Talagrand type $(s,\psi_{s,\delta})$. They also proved that if $X$ has martingale type $s$, then $X$ has Talagrand type $(s,\psi_{s,s/2})$. This led them to ask whether every Banach space of Rademacher type $s$ also has Talagrand type $(s,\psi_{s,s/2})$.

When $s=2$, the range $0<\delta<s/2$ already gives every $0<\delta<1$. Our main result gives the corresponding range $0<\delta<1$ for every $1<s<2$. In particular, it answers affirmatively their question concerning the exponent $\delta=s/2$ and, more generally, gives Talagrand type $(s,\psi_{s,\delta})$ for every $s/2\le\delta<1$.

Cordero-Erausquin and Eskenazis \cite{CorderoEskenazisLogSob} also proved the related logarithmic Sobolev inequality
$$
\|f-\mathbb{E}f\|_{L_s(\log L)^{s/2}(\Omega_n,X)}
\lesssim_X
\left(
\sum_{j=1}^n
\|D_jf\|_{L^s(\Omega_n,X)}^s
\right)^{1/s}
$$
for Banach spaces $X$ of Rademacher type $s$, where $1<s\le2$. Here the logarithmic exponent $s/2$ appears on the left-hand side, whereas in the Talagrand type inequality considered in this paper the logarithmic improvement appears in the norms of the derivatives.

The main estimate behind our result is the follwoing.

\begin{theorem}\label{thm:integral-estimate}
Let $1<s\le2$, and let $X$ be a Banach space with Rademacher type $s$.
Then for every $n\in\mathbb{N}$ and every $f:\Omega_n\to X$,
$$
\|f-\mathbb{E}f\|_{L^s(\Omega_n,X)}
\le
T_s^R(X)
\int_0^1
(1-\rho)^{1/s-1}
\left(
\sum_{j=1}^n
\|D_jf\|_{L^{1+\rho(s-1)}(\Omega_n,X)}^s
\right)^{1/s}
\,d\rho.
$$
\end{theorem}

Combining Theorem~\ref{thm:integral-estimate} with an appropriate Orlicz embedding yields the full range $0<\delta<1$.

\begin{theorem}\label{thm:rademacher-talagrand}
Let $1<s\le 2$, $0<\delta<1$, and let $X$ be a Banach space with Rademacher type $s$.
Then $X$ has Talagrand type $(s,\psi_{s,\delta})$. More precisely, there exists
a constant $C_{s,\delta}<\infty$ such that for every $n\in\mathbb{N}$ and every
$f:\Omega_n\to X$,
$$
\|f-\mathbb{E}f\|_{L^s(\Omega_n,X)}
\le
C_{s,\delta}T_s^R(X)
\left(
\sum_{j=1}^n
\|D_jf\|_{L_{\psi_{s,\delta}}(\Omega_n,X)}^s
\right)^{1/s}.
$$
\end{theorem}

The improvement over the range obtained by Cordero-Erausquin and Eskenazis \cite{CorderoEskenazis} is illustrated in Figure~\ref{fig:talagrand-ranges}. For $1<s<2$, our result covers the range $s/2\leq\delta<1$ and, in particular, answers their question at $\delta=s/2$. The endpoint $\delta=1$ remains open.

\begin{figure}[H]
\centering

\definecolor{newfill}{HTML}{EFB9B9}
\definecolor{oldfill}{HTML}{DCEAF7}
\definecolor{newink}{HTML}{A84545}
\definecolor{knownblue}{HTML}{286AA6}
\definecolor{diagonal}{HTML}{343A40}
\scalebox{0.70}{%
\begin{tikzpicture}[
    x=9.3cm, y=9.3cm,
    font=\small,
    line cap=round,
    line join=round,
    >=Stealth
]
% Coordinates are (s-1, delta).
\fill[oldfill]
    (0,0)--(1,0)--(1,1)--(0,.5)--cycle;
\fill[newfill]
    (0,.5)--(1,1)--(0,1)--cycle;

% Boundary lines.
\draw[knownblue,line width=1.7pt]
    (1,0)--(1,1);
\draw[dash pattern=on 3pt off 3pt,line width=.7pt]
    (0,1)--(1,1);
\draw[diagonal,line width=1.4pt]
    (0,.5)--(1,1);

% Coordinate axes.
\draw[->,line width=.65pt] (0,0)--(1.055,0);
\draw[->,line width=.65pt] (0,0)--(0,1.065);

\node[anchor=west,font=\Large]
    at (1.075,0) {$s$};
\node[anchor=south,font=\Large]
    at (0,1.075) {$\delta$};

% Tick marks.
\foreach \y/\lab in {0/0,.5/{\frac12},1/1} {
    \draw[line width=.55pt] (-.012,\y)--(0,\y);
    \node[anchor=east] at (-.025,\y) {$\lab$};
}
\foreach \x/\lab in {0/1,1/2} {
    \draw[line width=.55pt] (\x,0)--(\x,-.012);
    \node[anchor=north] at (\x,-.025) {$\lab$};
}

% Remaining question.
\node[anchor=south] at (.52,1.025)
    {Remaining problem: $\delta=1$};

% Our result.
\node[newink,font=\normalsize]
    at (.265,.885) {Our result};
\node[newink,font=\large] at (.265,.807)
    {$\displaystyle \frac{s}{2}\leq\delta<1$};

% Earlier result.
\node[align=center,font=\normalsize] at (.51,.362)
    {Cordero-Erausquin\\and Eskenazis};
\node[font=\large] at (.51,.265)
    {$\displaystyle 0<\delta<\frac{s}{2}$};

% The answered question.
\node[
    diagonal,
    rotate=26.565,
    anchor=south,
    font=\normalsize
] at (.66,.835) {$\delta=s/2$};

\node[diagonal,align=center,font=\footnotesize]
    at (.62,.63)
    {Question of\\
     Cordero-Erausquin and Eskenazis\\
     (answered here)};

\draw[->,diagonal,line width=.65pt]
    (.64,.715)--(.67,.825);

% Known s=2 case.
\node[
    knownblue,
    rotate=90,
    anchor=north,
    align=center
] at (1.025,.51)
    {Known at $s=2$: $0<\delta<1$\\
     Cordero-Erausquin and Eskenazis};

% Open endpoints.
\draw[diagonal,fill=white,line width=.8pt]
    (0,.5) circle[radius=1.7pt];
\draw[black,fill=white,line width=.8pt]
    (1,1) circle[radius=1.7pt];
\draw[knownblue,fill=white,line width=.8pt]
    (1,0) circle[radius=1.7pt];

\end{tikzpicture}%
}

\caption{Parameter ranges for the implication from
Rademacher type $s$ to Talagrand type
$(s,\psi_{s,\delta})$, where $1<s\leq2$.}
\label{fig:talagrand-ranges}
\end{figure}
\FloatBarrier

The paper is organized as follows. In Section 2, we collect the preliminary material, including the Orlicz estimate and the estimate for conditional expectations over random subsets used in the proof. In Section 3, we prove Theorems~\ref{thm:integral-estimate} and~\ref{thm:rademacher-talagrand}.

\section{Preliminaries}

Let $(X,\|\cdot\|_X)$ be a Banach space, and let
$\Omega_n=\{-1,1\}^n$ denote the $n$-dimensional Hamming cube equipped with
the uniform probability measure. We write $[n]=\{1,\ldots,n\}$. For a function $f:\Omega_n\to X$, we write
$$
\mathbb Ef
=
\frac{1}{2^n}\sum_{\varepsilon\in\Omega_n}f(\varepsilon),
$$
and, for $1\le p<\infty$,
$$
\|f\|_{L^p(\Omega_n,X)}
=
\left(
\mathbb E\|f(\varepsilon)\|_X^p
\right)^{1/p}.
$$
For $A\subseteq[n]$, we write
$\varepsilon_A=(\varepsilon_j)_{j\in A}$. Every function $f:\Omega_n\to X$ admits a Walsh expansion
$$
f(\varepsilon)
=
\sum_{S\subseteq[n]}\widehat f(S)W_S(\varepsilon),
$$
where
$$
W_S(\varepsilon)
=
\prod_{j\in S}\varepsilon_j,
\qquad
\widehat f(S)
=
\mathbb E\big[f(\varepsilon)W_S(\varepsilon)\big],
$$
with $W_\varnothing=1$. For a subset $A\subseteq[n]$, let $\mathcal F_A=\sigma(\varepsilon_j:j\in A)$
be the $\sigma$-algebra generated by the coordinates indexed by $A$, and write
$$
E_Af
=
\mathbb E[f\mid\mathcal F_A]
$$
for the corresponding conditional expectation. Then clearly
$$
E_Af=\sum_{S\subseteq A}\widehat f(S)W_S(\varepsilon).
$$

If $A\subseteq[n]$ and $j\notin A$, we define the increment associated with revealing the $j$-th coordinate by
$$
d_j^Af
=
E_{A\cup\{j\}}f-E_Af.
$$
Notice that for $j\notin A$, we have
$$d_j^Af=E_{A\cup\{j\}}D_jf.$$
Moreover,
$$d_j^Af(\varepsilon)=\varepsilon_jg_j^A(\varepsilon_A),$$
where
$$g_j^A(\varepsilon_A)=\sum_{S\subseteq A}\ \widehat{f}(S\cup\{j\})W_S(\varepsilon).$$

We shall also use the following estimate, proved in \cite{AreshidzeEnflo}.

\begin{lemma}\label{lem:reveal-type}
Let $1\le p\le2$, and suppose that $X$ has Rademacher type $p.$ Then for every
$A\subseteq[n]$,
$$\left\|
\sum_{j\notin A}d_j^Af
\right\|_{L^p(\Omega_n,X)}
\le
T_p^R(X)
\left(
\sum_{j\notin A}
\|d_j^Af\|_{L^p(\Omega_n,X)}^p
\right)^{1/p}.$$
\end{lemma}

For $0\le \rho\le1$, for every $f:\Omega_n\to X$ let $T_\rho$ denote the noise operator, 
$$T_\rho f=\sum_{S\subseteq[n]}\rho^{|S|}\widehat f(S)W_S.$$
Let $A_{\rho}\subseteq [n]$ be the random subset obtained by including each coordinate independently with probability $\rho$. Thus, for every $A\subseteq [n]$,

$$\mathbb{P}(A_\rho=A)=\rho^{|A|}(1-\rho)^{n-|A|}.$$
 
Then we have the following representation:

$$T_{\rho}f=\mathbb{E}_{A_{\rho}}E_{A_\rho}f$$

Indeed, 

$$\begin{aligned}
\mathbb{E}_{A_{\rho}}E_{A_\rho}f&=\sum_{A\subseteq [n]}\rho^{|A|}(1-\rho)^{n-|A|}E_{A}f\\
&=\sum_{S\subseteq [n]}\widehat f(S)W_S(\varepsilon)\sum_{\substack{A\subseteq[n]\\ S\subseteq A}}\rho^{|A|}(1-\rho)^{n-|A|}
\end{aligned}$$
Since 
$$\sum_{\substack{A\subseteq[n]\\ S\subseteq A}}\rho^{|A|}(1-\rho)^{n-|A|}=\mathbb{P}(S\subseteq A_{\rho})=\rho^{|S|}$$
we obtain 
$$
\mathbb{E}_{A_{\rho}}E_{A_\rho}f=\sum_{S\subseteq [n]}\rho^{|S|}\widehat f(S)W_S(\varepsilon)=T_{\rho}f.
$$

\begin{lemma}\label{lem:noise-derivative}
Let $0\le \rho<1$. Then, for every $f:\Omega_n\to X$,

$$\frac{d}{d\rho}T_{\rho}f=\frac{1}{1-\rho}\mathbb{E}_{A_\rho}\sum_{j\notin A_\rho}d_{j}^{A_\rho}f.$$
\end{lemma}

\begin{proof}
It is enough to verify the identity on the Walsh functions $W_S$. Fix $S\subseteq[n]$. Since
$$
T_{\rho}W_{S}=\rho^{|S|}W_S,
$$
we have
$$
\frac{d}{d\rho}T_{\rho}W_{S}=|S|\rho^{|S|-1}W_S.
$$
On the other hand, for $j\notin A_\rho$,
$$
d_{j}^{A_\rho}W_S
=
E_{A_\rho\cup\{j\}}W_S-E_{A_\rho}W_S.
$$
This is nonzero precisely when $j\in S$ and $S\setminus\{j\}\subseteq A_\rho$. In this case,
$$
d_{j}^{A_{\rho}}W_S=W_S.
$$

Now, since
$$
\mathbb{P}\bigl(j\notin A_\rho,\,
S\setminus\{j\}\subseteq A_\rho\bigr)
=
(1-\rho)\rho^{|S|-1},
$$
we obtain
$$
\mathbb{E}_{A_\rho}
\sum_{j\notin A_\rho}
d_{j}^{A_\rho}W_{S}
=
\sum_{j\in S}
(1-\rho)\rho^{|S|-1}W_S
=
|S|(1-\rho)\rho^{|S|-1}W_S.
$$
Dividing both sides by $1-\rho$, we obtain
$$
\frac{d}{d\rho}T_{\rho}W_{s}
=
\frac{1}{1-\rho}
\mathbb{E}_{A_\rho}
\sum_{j\notin A_\rho}
d_{j}^{A_\rho}W_{S}.
$$
The case $S=\varnothing$ is immediate, since both sides vanish. By linearity, the identity holds for every $f:\Omega_n\to X$.
\end{proof}

\begin{lemma}\label{lem:random-restriction}
Let $1<q\le2$ and $0\le\rho\le1$. Then for every Banach space $X$ and every function $h:\Omega_n\to X$,
$$
\left(
\mathbb{E}_{A_\rho}
\|E_{A_\rho}h\|_{L^q(\Omega_n,X)}^{q}
\right)^{1/q}
\le
\|h\|_{L^{r}(\Omega_n,X)},
$$
where $r=1+\rho(q-1)$.
\end{lemma}
 
\begin{proof}
We first consider the scalar case. Let $h:\Omega_n\to\mathbb{R}$. The endpoint cases $\rho=0$ and $\rho=1$ are immediate, since $A_\rho=\varnothing$ and $A_\rho=[n]$, respectively. Thus, assume $0<\rho<1$.

For $n=1$, we have
$$
A_\rho=
\begin{cases}
\{1\}, & \text{with probability } \rho,\\
\varnothing, & \text{with probability } 1-\rho.
\end{cases}
$$
Hence,
$$
E_{\{1\}}h=h,
\qquad
E_{\varnothing}h=\mathbb{E}h.
$$
Therefore,
$$
\mathbb{E}_{A_\rho}\|E_{A_\rho}h\|_{L^q(\Omega_1)}^{q}
=
\rho\|h\|_{L^q(\Omega_1)}^q
+
(1-\rho)|\mathbb{E}h|^q.
$$

Let $a=h(-1)$ and $b=h(1)$. Then the inequality we want to prove becomes
$$
\left(
\rho\frac{|a|^q+|b|^q}{2}
+
(1-\rho)\left|\frac{a+b}{2}\right|^q
\right)^{1/q}
\le
\left(
\frac{|a|^r+|b|^r}{2}
\right)^{1/r}.
$$
Since $|a+b|\le |a|+|b|$, it is enough to prove
$$
\left(
\rho\frac{|a|^q+|b|^q}{2}
+
(1-\rho)\left(\frac{|a|+|b|}{2}\right)^q
\right)^{1/q}
\le
\left(
\frac{|a|^r+|b|^r}{2}
\right)^{1/r}.
$$
Thus, we may replace $h$ by $|h|$, and therefore assume that $a,b\ge0$. If $a=b=0$, there is nothing to prove. Otherwise, set
$$
m=\frac{a+b}{2}>0.
$$
We may also assume that $a\le b$ and define
$$
\delta=\frac{b-a}{a+b}.
$$
Then $0\le\delta\le1$. Notice that
$$
a=m(1-\delta),
\qquad
b=m(1+\delta).
$$
Substituting these into the inequality and dividing both sides by $m$, we obtain
$$
\left(
\frac{\rho}{2}(1-\delta)^q
+
\frac{\rho}{2}(1+\delta)^q
+
(1-\rho)
\right)^{1/q}
\le
\left(
\frac{(1-\delta)^r+(1+\delta)^r}{2}
\right)^{1/r}.
$$

Now set
$$
1-\varepsilon=\rho,
\qquad
\lambda_1=r,
\qquad
\lambda_2'=q.
$$
Then our inequality becomes
$$
\left(
\frac{1-\varepsilon}{2}(1-\delta)^{\lambda_2'}
+
\frac{1-\varepsilon}{2}(1+\delta)^{\lambda_2'}
+
\varepsilon
\right)^{1/\lambda_2'}
\le
\left(
\frac{(1-\delta)^{\lambda_1}+(1+\delta)^{\lambda_1}}{2}
\right)^{1/\lambda_1},
$$
which is exactly the inequality in \cite{NairWang}. It remains only to check that our parameters satisfy their assumptions. Since $\lambda_2'=q$, we have
$$
\lambda_2=\frac{q}{q-1},
$$
and therefore
$$
\lambda_2-1=\frac{1}{q-1}.
$$
Since $r=1+\rho(q-1)$, we obtain
$$
(\lambda_1-1)(\lambda_2-1)
=
(r-1)\frac{1}{q-1}
=
\rho
=
1-\varepsilon.
$$
Also, because $1<q\le2$,
$$
\lambda_2=\frac{q}{q-1}\ge2.
$$
Thus, our parameters satisfy the assumptions of their result, and the case $n=1$ follows.

Now assume that the result has been proved for $n-1$, and let us prove it for $n$. Set
$$
\varepsilon=(\varepsilon',\varepsilon_n),
$$
where $\varepsilon'\in\Omega_{n-1}$, and define
$$
h_{-}(\varepsilon')=h(\varepsilon',-1),
\qquad
h_{+}(\varepsilon')=h(\varepsilon',1).
$$
Let $B_\rho\subseteq[n-1]$ be the random set defined as above, and define $A_\rho\subseteq[n]$ by
$$
A_\rho=
\begin{cases}
B_\rho\cup\{n\}, & \text{with probability } \rho,\\
B_\rho, & \text{with probability } 1-\rho,
\end{cases}
$$
where the decision to include $n$ is independent of $B_\rho$.

If $A_\rho=B_\rho\cup\{n\}$, then
$$
E_{B_\rho\cup\{n\}}h(\varepsilon',-1)
=
E_{B_\rho}h_{-}(\varepsilon'),
$$
and
$$
E_{B_\rho\cup\{n\}}h(\varepsilon',1)
=
E_{B_\rho}h_{+}(\varepsilon').
$$
If $n\notin A_\rho$, then $A_\rho=B_\rho$, and therefore
$$
E_{B_\rho}h(\varepsilon',-1)
=
E_{B_\rho}h(\varepsilon',1)
=
\frac{
E_{B_\rho}h_{-}(\varepsilon')
+
E_{B_\rho}h_{+}(\varepsilon')
}{2}.
$$

Notice that
$$
\begin{aligned}
\mathbb{E}_{A_{\rho}}\|E_{A_{\rho}}h\|_{L^q(\Omega_n)}^q
&=
\mathbb{E}_{A_{\rho}}\mathbb{E}_{\varepsilon'}
\left(
\frac{
|E_{A_{\rho}}h(\varepsilon',-1)|^q
+
|E_{A_{\rho}}h(\varepsilon',1)|^q
}{2}
\right)\\
&=
\mathbb{E}_{B_{\rho}}\mathbb{E}_{\varepsilon'}
\left[
\frac{\rho}{2}
\left(
|E_{B_{\rho}\cup\{n\}}h(\varepsilon',-1)|^q
+
|E_{B_{\rho}\cup\{n\}}h(\varepsilon',1)|^q
\right)
\right.\\
&\qquad\qquad\left.
+
\frac{1-\rho}{2}
\left(
|E_{B_{\rho}}h(\varepsilon',-1)|^q
+
|E_{B_{\rho}}h(\varepsilon',1)|^q
\right)
\right].
\end{aligned}
$$
Using the identities above, we obtain
$$
\begin{aligned}
\mathbb{E}_{A_{\rho}}\|E_{A_{\rho}}h\|_{L^q(\Omega_n)}^q
&=
\mathbb{E}_{B_{\rho}}\mathbb{E}_{\varepsilon'}
\left[
\frac{\rho}{2}
\left(
|E_{B_{\rho}}h_{-}(\varepsilon')|^q
+
|E_{B_{\rho}}h_{+}(\varepsilon')|^q
\right)
\right.\\
&\qquad\qquad\left.
+
(1-\rho)
\left|
\frac{
E_{B_{\rho}}h_{-}(\varepsilon')
+
E_{B_{\rho}}h_{+}(\varepsilon')
}{2}
\right|^q
\right].
\end{aligned}
$$
Applying the $n=1$ inequality gives
$$
\begin{aligned}
\mathbb{E}_{A_{\rho}}\|E_{A_{\rho}}h\|_{L^q(\Omega_n)}^q
&\le
\mathbb{E}_{B_{\rho}}\mathbb{E}_{\varepsilon'}
\left[
\left(
\frac{
|E_{B_{\rho}}h_{-}(\varepsilon')|^r
+
|E_{B_{\rho}}h_{+}(\varepsilon')|^r
}{2}
\right)^{q/r}
\right].
\end{aligned}
$$
Taking the $1/q$ power, we obtain
$$
\begin{aligned}
\left(
\mathbb{E}_{A_{\rho}}\|E_{A_{\rho}}h\|_{L^q(\Omega_n)}^q
\right)^{1/q}
&\le
\left(
\mathbb{E}_{B_{\rho}}\mathbb{E}_{\varepsilon'}
\left(
\frac{
|E_{B_{\rho}}h_{-}(\varepsilon')|^r
+
|E_{B_{\rho}}h_{+}(\varepsilon')|^r
}{2}
\right)^{q/r}
\right)^{1/q}.
\end{aligned}
$$
Since $r=1+\rho(q-1)\le q$, we have $q/r\ge1$, so Minkowski's inequality gives
$$
\begin{aligned}
\left(
\mathbb{E}_{A_{\rho}}\|E_{A_{\rho}}h\|_{L^q(\Omega_n)}^q
\right)^{1/q}
&\le
\left(
\frac{1}{2}
\left(
\mathbb{E}_{B_{\rho}}\mathbb{E}_{\varepsilon'}
|E_{B_{\rho}}h_{-}(\varepsilon')|^q
\right)^{r/q}
\right.\\
&\qquad\left.
+
\frac{1}{2}
\left(
\mathbb{E}_{B_{\rho}}\mathbb{E}_{\varepsilon'}
|E_{B_{\rho}}h_{+}(\varepsilon')|^q
\right)^{r/q}
\right)^{1/r}\\
&=
\left(
\frac{1}{2}
\left(
\mathbb{E}_{B_{\rho}}
\|E_{B_{\rho}}h_{-}\|_{L^q(\Omega_{n-1})}^{q}
\right)^{r/q}
\right.\\
&\qquad\left.
+
\frac{1}{2}
\left(
\mathbb{E}_{B_{\rho}}
\|E_{B_{\rho}}h_{+}\|_{L^q(\Omega_{n-1})}^{q}
\right)^{r/q}
\right)^{1/r}.
\end{aligned}
$$
By the induction hypothesis,
$$
\begin{aligned}
\left(
\mathbb{E}_{A_{\rho}}\|E_{A_{\rho}}h\|_{L^q(\Omega_n)}^q
\right)^{1/q}
&\le
\left(
\frac{1}{2}\|h_{-}\|_{L^r(\Omega_{n-1})}^{r}
+
\frac{1}{2}\|h_{+}\|_{L^r(\Omega_{n-1})}^{r}
\right)^{1/r}\\
&=
\|h\|_{L^r(\Omega_n)}.
\end{aligned}
$$

Now let $h:\Omega_n\to X$, where $X$ is a Banach space. We want to prove
$$
\left(
\mathbb{E}_{A_\rho}
\|E_{A_\rho}h\|_{L^q(\Omega_n,X)}^{q}
\right)^{1/q}
\le
\|h\|_{L^{r}(\Omega_n,X)}.
$$
Recall that
$$
E_{A_{\rho}}h(\varepsilon)
=
\mathbb{E}_{\eta_{A_{\rho}^c}}
h(\varepsilon_{A_{\rho}},\eta_{A_{\rho}^c}).
$$
Therefore,
$$
\begin{aligned}
\|E_{A_{\rho}}h(\varepsilon)\|_{X}
&=
\left\|
\mathbb{E}_{\eta_{A_{\rho}^c}}
h(\varepsilon_{A_{\rho}},\eta_{A_{\rho}^c})
\right\|_{X}\\
&\le
\mathbb{E}_{\eta_{A_{\rho}^c}}
\|h(\varepsilon_{A_{\rho}},\eta_{A_{\rho}^c})\|_{X}\\
&=
E_{A_\rho}(\|h\|_{X})(\varepsilon).
\end{aligned}
$$
Hence,
$$
\begin{aligned}
\left(
\mathbb{E}_{A_{\rho}}
\|E_{A_\rho}h\|_{L^q(\Omega_n,X)}^{q}
\right)^{1/q}
&=
\left(
\mathbb{E}_{A_{\rho}}\mathbb{E}_{\varepsilon}
\|E_{A_\rho}h(\varepsilon)\|_{X}^{q}
\right)^{1/q}\\
&\le
\left(
\mathbb{E}_{A_{\rho}}\mathbb{E}_{\varepsilon}
\left[
E_{A_\rho}(\|h\|_{X})(\varepsilon)
\right]^{q}
\right)^{1/q}\\
&=
\left(
\mathbb{E}_{A_{\rho}}
\|E_{A_\rho}(\|h\|_{X})\|_{L^q(\Omega_n)}^{q}
\right)^{1/q}.
\end{aligned}
$$
Define the scalar function $g(\varepsilon)=\|h(\varepsilon)\|_{X}$. Then, by the scalar inequality,
$$
\begin{aligned}
\left(
\mathbb{E}_{A_{\rho}}
\|E_{A_\rho}h\|_{L^q(\Omega_n,X)}^{q}
\right)^{1/q}
&\le
\|g\|_{L^r(\Omega_n)}
=
\|h\|_{L^r(\Omega_n,X)}.
\end{aligned}
$$
This completes the proof.
\end{proof}

We shall also need the following Orlicz embedding.

\begin{lemma}\label{lem:orlicz-embedding}
Let $1<s\le 2,\, 0<\delta\le 1$ and $R\ge e$. Suppose that
$\psi_{s,\delta}:[0,\infty)\to[0,\infty)$ is a Young function satisfying
$$
\psi_{s,\delta}(t)=t^s\log^{-\delta}(e+t) \quad
\text{for}\quad t\ge R.
$$
Then there exists a constant $C_R>0$, depending only on $R$, such that for every Banach space $X$, every $n\ge1$, every $1\le r<s$, and every $h:\Omega_n\to X$,
$$
\|h\|_{L^r(\Omega_n,X)}
\le
\frac{C_R}{(s-r)^{\delta/s}}
\|h\|_{L_{\psi_{s,\delta}}(\Omega_n,X)}.
$$
\end{lemma}

\begin{proof}
Set $a=s-r$. Then $0<a\le s-1\le1$. By homogeneity, it is enough to consider functions $g:\Omega_n\to X$ satisfying
$$
\|g\|_{L_{\psi_{s,\delta}}(\Omega_n,X)}=1.
$$
Under this normalization, it suffices to prove that
$$
\|g\|_{L^r(\Omega_n,X)}
\le
\frac{C_R}{(s-r)^{\delta/s}}.
$$
By the definition of the Luxemburg norm,
$$
\mathbb{E}_{\varepsilon}\psi_{s,\delta}(\|g(\varepsilon)\|_X)\le 1.
$$

For $t\ge R$, since $r=s-a$, we can write
$$
t^r=\psi_{s,\delta}(t]t^{-a}\log^{\delta}(e+t).
$$
Since $t\ge R\ge e$, we have $\log(e+t)\le 2\log t$. Consequently,
$$
t^{-a}\log^{\delta}(e+t)
\le
2^{\delta}t^{-a}(\log t)^{\delta}.
$$
Let $u=\log t$. Then
$$
t^{-a}(\log t)^{\delta}=e^{-au}u^{\delta}.
$$
The function $F(u)=e^{-au}u^{\delta}$ on $(0,\infty)$ attains its maximum at $u=\frac{\delta}{a}$. Thus,
$$
\sup_{u\ge\log R}e^{-au}u^{\delta}
\le
e^{-\delta}\left(\frac{\delta}{a}\right)^{\delta}.
$$
Since $0<\delta\le 1$, we have $0<\frac{\delta}{e}<1$, and hence
$$
\sup_{u\ge\log R}e^{-au}u^{\delta}
\le
a^{-\delta}.
$$
It follows that
$$
t^r
\le
2^{\delta}a^{-\delta}\psi_{s,\delta}(t)
\le
2a^{-\delta}\psi_{s,\delta}(t).
$$

On the set where $\|g(\varepsilon)\|_X<R$, we simply have
$$
\|g(\varepsilon)\|_X^r<R^r.
$$
On the complementary set, the estimate above gives
$$
\|g(\varepsilon)\|_X^r
\le
2a^{-\delta}\psi_{s,\delta}(\|g(\varepsilon)\|_X).
$$
Therefore,
$$
\begin{aligned}
\|g\|_{L^r(\Omega_n,X)}^r
&=
\mathbb{E}_{\varepsilon}\|g(\varepsilon)\|_X^r\\
&\le
R^r+
2a^{-\delta}
\mathbb{E}_{\varepsilon}\psi_{s,\delta}(\|g(\varepsilon)\|_X)\\
&\le
R^r+2a^{-\delta}.
\end{aligned}
$$
Since $0<a\le1$, we have $a^{-\delta}\ge1$, and since $R>1$ and $r<s\le2$,
$$
R^r\le R^2.
$$
Thus,
$$
\|g\|_{L^r(\Omega_n,X)}^r
\le
(R^2+2)a^{-\delta}.
$$
Taking the $r$-th root yields
$$
\|g\|_{L^r(\Omega_n,X)}
\le
(R^2+2)^{1/r}a^{-\delta/r}.
$$
Since $s=r+a$,
$$
\frac{\delta}{r}
=
\frac{\delta}{s}+\frac{\delta(s-r)}{rs}=\frac{\delta}{s}+\frac{\delta a}{rs},
$$
and consequently
$$
a^{-\delta/r}
=
a^{-\delta/s}a^{-\delta a/(rs)}.
$$
Notice that
$$
\sup_{0<a\le1}a\log\frac1a=\frac1e.
$$
Since $\frac{\delta}{rs}\le 1$, we obtain
$$
a^{-\delta a/(rs)}=\exp\left(\frac{\delta}{rs}a\log \frac{1}{a}\right)\le e^{1/e}.
$$
Thus,
$$
a^{-\delta/r}\le e^{1/e}a^{-\delta/s}.
$$

Combining these estimates, we obtain
$$
\|g\|_{L^r(\Omega_n,X)}
\le
e^{1/e}(R^2+2)a^{-\delta/s}=\frac{e^{1/e}(R^2+2)}{(s-r)^{\delta/s}}.
$$
This proves the lemma.
\end{proof}

\section{Proofs of the Main Results}

\begin{proof}[Proof of Theorem~\ref{thm:integral-estimate}]
Since $T_1f=f$ and $T_0f=\mathbb{E}f$, we have

$$
f-\mathbb{E}f
=
\int_0^1\frac{d}{d\rho}T_\rho f\,d\rho.
$$

Therefore, by Minkowski's integral inequality,

$$
\|f-\mathbb{E}f\|_{L^s(\Omega_n,X)}
\le
\int_0^1
\left\|
\frac{d}{d\rho}T_\rho f
\right\|_{L^s(\Omega_n,X)}
\,d\rho.
$$

Thus, it is enough to estimate the integrand on the right-hand side.
Fix $0\le \rho<1$ and put $r=1+\rho(s-1)$. By Lemma~\ref{lem:noise-derivative}, we have 

$$
\frac{d}{d\rho}T_{\rho}f
=
\frac{1}{1-\rho}\mathbb{E}_{A_\rho}
\sum_{j\notin A_\rho}d_{j}^{A_\rho}f.
$$

Therefore, by the triangle inequality,

$$
\left\|
\frac{d}{d\rho}T_{\rho}f
\right\|_{L^s(\Omega_n,X)}
\le
\frac{1}{1-\rho}
\mathbb{E}_{A_\rho}
\left\|
\sum_{j\notin A_\rho}d_{j}^{A_\rho}f
\right\|_{L^s(\Omega_n,X)}.
$$

Applying Lemma~\ref{lem:reveal-type} for every fixed realization of $A_\rho$, and Jensen's inequality, we obtain

$$
\begin{aligned}
\left\|
\frac{d}{d\rho}T_{\rho}f
\right\|_{L^s(\Omega_n,X)}
&\le
\frac{T_{s}^{R}(X)}{1-\rho}
\mathbb{E}_{A_\rho}
\left(
\sum_{j\notin A_{\rho}}
\|d_j^{A_{\rho}}f\|_{L^s(\Omega_n,X)}^s
\right)^{1/s}\\
&\le
\frac{T_{s}^{R}(X)}{1-\rho}
\left(
\sum_{j=1}^{n}
\mathbb{E}_{A_\rho}
\left[
\mathbf{1}_{\{j\notin A_\rho\}}
\|d_j^{A_{\rho}}f\|_{L^s(\Omega_n,X)}^s
\right]
\right)^{1/s}.
\end{aligned}
$$

Now, using the definition,

$$
\begin{aligned}
\mathbb{E}_{A_\rho}
\left[
\mathbf{1}_{\{j\notin A_\rho\}}
\|d_j^{A_{\rho}}f\|_{L^s(\Omega_n,X)}^s
\right]
&=
\sum_{A\subseteq[n]\setminus\{j\}}
\rho^{|A|}(1-\rho)^{n-|A|}
\|d_j^{A}f\|_{L^s(\Omega_n,X)}^s\\
&=
(1-\rho)
\sum_{A\subseteq[n]\setminus\{j\}}
\rho^{|A|}(1-\rho)^{(n-1)-|A|}
\|d_j^{A}f\|_{L^s(\Omega_n,X)}^s.
\end{aligned}
$$

For fixed $j\in[n]$, write

$$
D_jf(\varepsilon)=\varepsilon_j g_j(\varepsilon),
$$

where

$$
g_j(\varepsilon)
=
\sum_{B\subseteq[n]\setminus\{j\}}
\widehat f(B\cup\{j\})W_B(\varepsilon).
$$

Then $g_j$ does not depend on $\varepsilon_j$. Moreover, for every
$A\subseteq[n]\setminus\{j\}$,

$$
d_j^Af
=
E_{A\cup\{j\}}D_jf
=
\varepsilon_jE_Ag_j.
$$

Hence,
$$
\|d_j^Af\|_{L^s(\Omega_n,X)}
=
\|E_Ag_j\|_{L^s(\Omega_{n-1},X)}.
$$
Therefore, by Lemma~\ref{lem:random-restriction} applied on the $(n-1)$-dimensional cube,

$$
\begin{aligned}
\mathbb{E}_{A_\rho}
\left[
\mathbf{1}_{\{j\notin A_\rho\}}
\|d_j^{A_{\rho}}f\|_{L^s(\Omega_n,X)}^s
\right]
&=
(1-\rho)
\sum_{A\subseteq[n]\setminus\{j\}}
\rho^{|A|}(1-\rho)^{(n-1)-|A|}
\|E_Ag_j\|_{L^s(\Omega_{n-1},X)}^s\\
&\le
(1-\rho)
\|g_j\|_{L^r(\Omega_{n-1},X)}^s.
\end{aligned}
$$

Since $D_jf(\varepsilon)=\varepsilon_jg_j(\varepsilon)$,

$$
\|g_j\|_{L^r(\Omega_{n-1},X)}
=
\|D_jf\|_{L^r(\Omega_n,X)}.
$$

Thus, we obtain

$$
\left\|
\frac{d}{d\rho}T_{\rho}f
\right\|_{L^s(\Omega_n,X)}
\le
T_{s}^{R}(X)(1-\rho)^{1/s-1}
\left(
\sum_{j=1}^{n}
\|D_jf\|_{L^r(\Omega_n,X)}^s
\right)^{1/s}.
$$

Now recall that $r=1+\rho(s-1)$, hence

$$
\left\|
\frac{d}{d\rho}T_{\rho}f
\right\|_{L^s(\Omega_n,X)}
\le
T_{s}^{R}(X)(1-\rho)^{1/s-1}
\left(
\sum_{j=1}^{n}
\|D_jf\|_{L^{1+\rho(s-1)}(\Omega_n,X)}^s
\right)^{1/s}.
$$

Putting this into the first inequality above, we obtain

$$
\|f-\mathbb{E}f\|_{L^s(\Omega_n,X)}
\le
T_{s}^{R}(X)
\int_{0}^{1}
(1-\rho)^{1/s-1}
\left(
\sum_{j=1}^{n}
\|D_jf\|_{L^{1+\rho(s-1)}(\Omega_n,X)}^s
\right)^{1/s}
\,d\rho.
$$

The proof is finished.
\end{proof}

We now combine the preceding estimate with the Orlicz embedding from Section~2 to prove Theorem~\ref{thm:rademacher-talagrand}.

\begin{proof}[Proof of Theorem~\ref{thm:rademacher-talagrand}]
By Theorem~\ref{thm:integral-estimate},
$$
\|f-\mathbb{E}f\|_{L^s(\Omega_n,X)}
\le
T_s^R(X)
\int_0^1
(1-\rho)^{1/s-1}
\left(
\sum_{j=1}^n
\|D_jf\|_{L^{1+\rho(s-1)}(\Omega_n,X)}^s
\right)^{1/s}
\,d\rho.
$$
Let $r=1+\rho(s-1)$. Then
$$
s-r=(s-1)(1-\rho).
$$
By Lemma~\ref{lem:orlicz-embedding},
$$
\|D_jf\|_{L^r(\Omega_n,X)}
\le
\frac{C_{R}}{((s-1)(1-\rho))^{\delta/s}}
\|D_jf\|_{L_{\psi_{s,\delta}}(\Omega_n,X)}.
$$
Therefore,
$$
\left(
\sum_{j=1}^n
\|D_jf\|_{L^{1+\rho(s-1)}(\Omega_n,X)}^s
\right)^{1/s}
\le
\frac{C_{R}}{((s-1)(1-\rho))^{\delta/s}}
\left(
\sum_{j=1}^n
\|D_jf\|_{L_{\psi_{s,\delta}}(\Omega_n,X)}^s
\right)^{1/s}.
$$
Hence,
$$
\|f-\mathbb{E}f\|_{L^s(\Omega_n,X)}
\le
\frac{C_{R}\,T_s^R(X)}{(s-1)^{\delta/s}}
\left(
\int_0^1
(1-\rho)^{(1-\delta)/s-1}\,d\rho
\right)
\left(
\sum_{j=1}^n
\|D_jf\|_{L_{\psi_{s,\delta}}(\Omega_n,X)}^s
\right)^{1/s}.
$$

Since $(1-\delta)/s>0$, the integral is finite. In fact,
$$
\int_0^1(1-\rho)^{(1-\delta)/s-1}\,d\rho
=\frac{s}{1-\delta}.
$$
Therefore,
$$
\|f-\mathbb{E}f\|_{L^s(\Omega_n,X)}
\le
\frac{sC_{R}}{(s-1)^{\delta/s}(1-\delta)}\,T_s^R(X)
\left(
\sum_{j=1}^n
\|D_jf\|_{L_{\psi_{s,\delta}}(\Omega_n,X)}^s
\right)^{1/s}.
$$
Thus, setting
$$
C_{s,\delta}=\frac{sC_{R}}{(s-1)^{\delta/s}(1-\delta)},
$$
we obtain
$$
\|f-\mathbb{E}f\|_{L^s(\Omega_n,X)}
\le
C_{s,\delta}T_s^R(X)
\left(
\sum_{j=1}^n
\|D_jf\|_{L_{\psi_{s,\delta}}(\Omega_n,X)}^s
\right)^{1/s}.
$$
Hence $X$ has Talagrand type $(s,\psi_{s,\delta})$ for every $1<s\le2$ and $0<\delta<1$.
\end{proof}

\begin{remark}
The remaining endpoint is $\delta=1$. Our argument does not reach this case, since the integral in the proof becomes
$$
\int_0^1(1-\rho)^{-1}\,d\rho,
$$
which diverges.

When $s=2$, Cordero-Erausquin and Eskenazis \cite{CorderoEskenazisLogSob} showed that this endpoint problem is equivalent to proving the weaker estimate
$$
\|f-\mathbb{E}f\|_{L^1(\Omega_n,X)}^2
\lesssim_X
\sum_{j=1}^n
\|D_jf\|_{L_2(\log L)^{-1}(\Omega_n,X)}^2
$$
for every Banach space $X$ of Rademacher type $2$.
\end{remark}

\section*{Acknowledgments}
The author is grateful to Paata Ivanisvili for introducing him to the problem and to Alexandros Eskenazis for helpful discussions. The author acknowledges the use of AI tools. All mathematical arguments and proofs in the final manuscript were checked and written by the author.

\end{document}